\documentclass{amsart}

\usepackage{amsfonts}
\usepackage{amsmath}
\usepackage{amssymb}
\usepackage{amsthm}
\usepackage{graphicx}
\usepackage{url}
\usepackage{arydshln}
\usepackage[latin1]{inputenc}
\usepackage{subfigure}

\renewcommand{\S}[1]{\ensuremath{\mathcal{S}_{#1}}} 
\newcommand{\E}[1]{\ensuremath{\mathcal{E}_{#1}}}   
\renewcommand{\O}[1]{\ensuremath{\mathcal{O}_{#1}}}   
\newcommand{\I}[1]{\ensuremath{\mathcal{I}_{#1}}}    

\newcommand{\Sav}[2]{\ensuremath{\mathcal{S}_{#1}(#2)}} 
\newcommand{\Eav}[2]{\ensuremath{\mathcal{E}_{#1}(#2)}}   
\newcommand{\Oav}[2]{\ensuremath{\mathcal{O}_{#1}(#2)}}   

\newcommand{\sav}[2]{\ensuremath{S_{#1}(#2)}} 
\newcommand{\eav}[2]{\ensuremath{E_{#1}(#2)}} 
\newcommand{\oav}[2]{\ensuremath{O_{#1}(#2)}} 

   \newcommand{\oi}{\sim}
   \newcommand{\we}{\equiv}
   \newcommand{\Swe}{\equiv_{\S{n}}}
   \newcommand{\ewe}{\equiv_{\E{n}}}
   
   \newcommand{\swe}{\ensuremath{\stackrel{s}{\equiv}_{\S{n}}}}
   \newcommand{\eswe}{\ensuremath{\stackrel{s}{\equiv}_{\E{n}}}}

\newcommand{\inv}{\ensuremath{\textsc{inv}}}

\newcommand{\sgn}{\ensuremath{\mathrm{sgn}}}

\newtheorem{theorem}{Theorem}

\newtheorem{conjecture}[theorem]{Conjecture}
\newtheorem{lemma}[theorem]{Lemma}

\newtheorem{algorithm}[theorem]{Algorithm}

\newif\ifgraph
\graphtrue

\title{Pattern avoidance by even permutations}

\thanks{Both authors partially supported by NSA award H98230-09-1-0014.}

\author{Andrew M.\ Baxter}
\address{Department of Mathematics, Rutgers University}
\email{baxter@math.rutgers.edu}
\author{Aaron D.\ Jaggard}
\address{DIMACS, Rutgers University}
\curraddr{Department of Computer Science, Colgate University and DIMACS, Rutgers University}
\email{adj@dimacs.rutgers.edu}

\begin{document}

\begin{abstract}
We study questions of even-Wilf-equivalence, the analogue of Wilf-equivalence when attention is restricted to pattern avoidance by permutations in the alternating group.  Although some Wilf-equivalence results break when considering even-Wilf-equivalence analogues, we prove that other Wilf-equivalence results continue to hold in the even-Wilf-equivalence setting.  In particular, we prove that $t(t-1)\cdots 321$ and $(t-1)(t-2)\cdots 21t$ are even-shape-Wilf-equivalent for odd $t$, paralleling a result (which held for all $t$) of Backelin, West, and Xin for shape-Wilf-equivalence.  This allows us to classify the symmetric group $\S{4}$, and to partially classify $\S{5}$ and $\S{6}$, according to even-Wilf-equivalence.  As with transition to involution-Wilf-equivalence, some---but not all---of the classical Wilf-equivalence results are preserved when we make the transition to even-Wilf-equivalence.
\end{abstract}

\maketitle

\section{Introduction}

In this paper we focus on questions of Wilf-equivalence when we count only the \emph{even} permutations (i.e., the members of the alternating subgroup) that avoid a particular pattern.  In particular, we are interested in which classical Wilf-equivalence results have parallels when we instead consider even-Wilf-equivalence.  These investigations parallel comparisons between classical Wilf-equivalence and involution-Wilf-equivalence seen in other work (e.g.,~\cite{Simion1985,Jaggard2002,Dukes2009,jm11ac}).

Given a permutation $\sigma=\sigma_1\cdots\sigma_k\in\S{k}$, $\pi=\pi_1\cdots\pi_n\in\S{n}$ is said to \emph{contain the pattern} $\tau\in S_k$ if there is some sequence of indices $i_1<\cdots<i_k$ such that $\pi_{i_1}\cdots\pi_{i_k}$ is order isomorphic to $\tau_1\cdots\tau_k$.  If $\pi$ does not contain $\tau$, then $\pi$ is said to \emph{avoid the pattern} $\tau$.  We let $\Sav{n}{\sigma}$ denote the set of permutations avoiding $\sigma$ and let $\sav{n}{\sigma}:=\#\Sav{n}{\sigma}$.  Two permutations $\sigma$ and $\tau$ are said to be [classically] \emph{Wilf-equivalent} if, for every positive integer $n$, $\sav{n}{\sigma}=\sav{n}{\tau}$; we then write $\sigma\we\tau$. 

A pair of indices $i<j$ forms an \emph{inversion} in permutation $\pi$ if $\pi_i > \pi_j$.  Let $\inv(\pi)$ denote the number of inversions in $\pi$, and let $\sgn(\pi) = (-1)^{\inv(\pi)}$ be the \emph{sign} of $\pi$.  If $\sgn(\pi)=1$ (i.e., $\inv(\pi)$ is even) we say that $\pi$ is \emph{even} and otherwise $\pi$ is \emph{odd}.  Let $\E{n}\subset\S{n}$ be the set of even permutations of length $n$ and $\O{n}\subset\S{n}$ be the set of odd permutations of length $n$.  Let $\Eav{n}{\sigma}=\Sav{n}{\sigma} \cap \E{n}$ be the set of even permutations avoiding $\sigma$ and let $\eav{n}{\sigma}:=\#\Eav{n}{\sigma}$, and similarly for $\O{n}(\sigma)$ and $\oav{n}{\sigma}$.  We say that two permutations $\sigma$ and $\tau$ are \emph{even-Wilf-equivalent} if $\eav{n}{\sigma}=\eav{n}{\tau}$ for all $n\geq 0$; we then write $\sigma\ewe\tau$.  When we need contrast with classical Wilf-equivalence, we denote classical Wilf-equivalence $\sigma\Swe\tau$.

Our main result, presented in Theorem~\ref{theJtFt}, is the even-shape-Wilf-equivalence of $J_t = t(t-1)\cdots 321$ and $F_t = (t-1)(t-2)\cdots 21t$ when $t$ is odd.  This parallels the analogous result for shape-Wilf-equivalence due to Backelin, West, and Xin~\cite{Backelin2007}, which held for all $t$.  As corollaries of our main result, we classify the permutations in $\S{4}$ according to $\ewe$ and give partial classifications of $\S{5}$ and $\S{6}$ according to $\ewe$; we also conjecture a number of other even-Wilf-equivalences that parallel known Wilf-equivalences.

\section{Equivalences via Symmetry}

In this section we present two useful lemmas connecting classical Wilf-equivalence to even-Wilf-equivalence.  First we exhibit a case that shows where it is is clear that $\sigma$ is \emph{not} even-Wilf-equivalent to $\tau$.

\begin{lemma}\label{lem:sgn}
 If $\sigma, \tau \in \S{k}$ but $\sgn(\sigma)\neq\sgn(\tau)$, then $\sigma\not\ewe\tau$.
\end{lemma}

\begin{proof}
 If $\sigma$ is even and $\tau$ is odd, then $\Eav{k}{\sigma}= \E{n} \setminus \{\sigma\}$ while $\Eav{k}{\tau}=\E{k}$.  Hence $\eav{k}{\sigma} = \eav{k}{\tau}-1$.
\end{proof}

Next we consider the trivial symmetries induced by the symmetry of the square.  Recall the \emph{reverse} of $\pi=\pi_1\pi_2\cdots\pi_n$ is the horizontal reflection of $\pi$, denoted
 \begin{equation*}
   {\pi^{r}:=\pi_n\pi_{n-1}\cdots\pi_1}.
  \end{equation*}
 Similarly, the \emph{complement} of $\pi\in \S{n}$ is the vertical reflection
  \begin{equation*}
    \pi^c:={(n+1-\pi_1)(n+1-\pi_2)\cdots(n+1-\pi_n)}.
   \end{equation*}
 The inverse of $\pi$ is denoted as usual by $\pi^{-1}$.  The following lemma summarizes how these reflections affect the sign of $\pi$.

\begin{lemma}\label{lem:sgnsymm}
 The sign of a permutation $\pi\in S_n$ is affected by reflections in the following ways:
 \begin{enumerate}
   \item[(a.)] $\sgn(\pi) = \sgn(\pi^r)$ if and only if $n\equiv 0,1 \pmod{4}$.
   \item[(b.)] $\sgn(\pi) = \sgn(\pi^c)$ if and only if $n\equiv 0,1 \pmod{4}$.
   \item[(c.)] $\sgn(\pi) = \sgn(\pi^{-1})$
 \end{enumerate}
\end{lemma}

\begin{proof}
  For each pair of indices $i<j$, $\pi_i > \pi_j$ if and only if $(\pi^r)_i < (\pi^r)_j$.  That is, the reversal map swaps the sites of inversions and non-inversions.  Therefore $\inv(\pi^r) = \binom{n}{2} - \inv(\pi)$.  Since $\binom{n}{2}$ is even if and only if $n\equiv 0,1 \pmod{4}$, part (a) is proven.  Part (b) is proven similarly since it is also the case that $\inv(\pi^c) = \binom{n}{2} - \inv(\pi)$.  Part (c) follows from the fact that for any permutation, $\inv(\pi) = \inv(\pi^{-1})$.
\end{proof}

In the classical case, $\sigma \Swe \sigma^{r}$, $\sigma\Swe\sigma^{c}$, and $\sigma\Swe\sigma^{-1}$.  Parts (a) and (b) of the lemma above, however, show that even-Wilf-equivalence for $\sigma$ and $\sigma^{r}$ is not guaranteed, and similarly for $\sigma^{c}$.  For example $123\not\ewe 321$, since $\eav{3}{123}=2$ and $\eav{3}{321}=3$.  Part (c) confirms, however, that we still have $\sigma \ewe \sigma^{-1}$.

The next lemma demonstrates that while we lose the equivalences from reversal and complement, we may use symmetric versions of any even-Wilf-equivalences discovered.

\begin{lemma}\label{lem2}
 If $\sigma \ewe \tau$ and $\sigma \Swe \tau$, then $\sigma^r \ewe \tau^r$ and $\sigma^c \ewe \tau^c$.
\end{lemma}

\begin{proof}
 We will prove $\sigma^r \ewe \tau^r$.  The proof for $\sigma^c \ewe \tau^c$ is analogous.  First observe that if $\sav{n}{\sigma}=\sav{n}{\tau}$ and $\eav{n}{\sigma}=\eav{n}{\tau}$, then $\oav{n}{\sigma}=\oav{n}{\tau}$.  We continue by cases.  If $n\equiv 0$ or $1\pmod{4}$, then
 \begin{equation*}
   \eav{n}{\sigma^r}=\eav{n}{\sigma}=\eav{n}{\tau}=\eav{n}{\tau^r},
 \end{equation*}
 where the first and third equalities follow from Lemma $\ref{lem:sgnsymm}$ and the second equality by our assumptions.  If $n\equiv 2$ or $3\pmod{4}$, then we see
 \begin{equation*}
   \eav{n}{\sigma^r}=\oav{n}{\sigma}=\oav{n}{\tau}=\eav{n}{\tau^r},
 \end{equation*}
 where again the first and third equalities follow from Lemma \ref{lem:sgnsymm} and the second equality by the observation above.
\end{proof}

It is worth stating the following lemma regarding the trivial equivalence classes for even-Wilf-equivalences.  Its proof is similar to those above and is left to the reader.

\begin{lemma}\label{lem3}
 For a pattern $\sigma$, we have the following trivial equivalences:
  \begin{itemize}
   \item $\sigma \ewe \sigma^{-1} \ewe \sigma^{rc} \ewe \left(\sigma^{-1}\right)^{rc}$
   \item $\sigma^r \ewe \sigma^{c} \ewe \left(\sigma^{-1}\right)^r \ewe \left(\sigma^{-1}\right)^c$
  \end{itemize}
\end{lemma}

\section{Short Patterns}

In this section we turn to the question of classifying patterns of a given length according to even-Wilf-equivalence.

Lemma \ref{lem:sgn} immediately implies $12\not\ewe 21$, which is the classification of $\S{2}$.  

Moving on to patterns of length three, we turn to Simion and Schmidt's observations in \cite{Simion1985}.  Their enumerations of $\eav{n}{\sigma}-\oav{n}{\sigma}$ for each $\sigma\in \S{3}$ imply the following
\begin{theorem}[Simion and Schmidt, 1985]
 There are two distinct even-Wilf-equivalence classes for patterns of length 3:
 \begin{itemize}
   \item  $123\ewe312\ewe231$
   \item  $321\ewe213\ewe132$
 \end{itemize}
\end{theorem}

This suggests that if $\sigma\Swe\tau$ and $\sgn(\sigma)=\sgn(\tau)$, then $\sigma\ewe\tau$.  This is not the case, however, as demonstrated by $1234\not\ewe4321$: $\eav{6}{1234}=258$, while $\eav{6}{4321}=255$.

To classify patterns of length 4 or more, we use tools developed in the next section.

\section{An Infinite Class of Non-trivial Equivalences}

In this section we discuss an extension of the celebrated ``prefix reversal'' result for classical Wilf-equivalence, as proven by Backelin, West, and Xin in \cite{Backelin2007}.  We follow and adapt their notation, aside from a change in convention: we reflect everything vertically.  Backelin et al. state their results in terms of (permutation) matrices avoiding other (permutation) matrices.  Hence the permutation 132, for example, is written as
\begin{equation*}
\begin{bmatrix}
1 & 0 & 0 \\
0 & 0 & 1 \\
0 & 1 & 0 \\
\end{bmatrix}
\end{equation*}
They then proceed to consider pattern avoidance in non-attacking rook placements in Young diagrams.  These rook placements correspond to permutation matrices with some of the southeast cells of the matrix absent.

We choose to illustrate our permutations as graphs of functions, hence our graph of 132 looks like Figure \ref{fig132}.

\ifgraph
\begin{figure}[hbt]
	\centering
		\includegraphics[width=.3\textwidth]{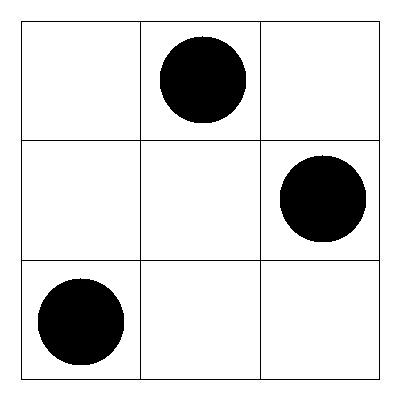}
	\caption{The graph of the permutation $132$}
	\label{fig132}
\end{figure}
\fi

As a result of this new convention, we orient our Young diagrams $\lambda=\lambda_1\geq\lambda_2\geq\cdots\geq\lambda_n$ such that the largest part $\lambda_1$ forms the \emph{bottom} row of cells (boxes), then $\lambda_2$ cells lie above this bottom layer, and so on as per the French custom.  Cells of $\lambda$ are indexed from the lower-left corner by rows and columns, so $(r,c)$ is the cell in the $r^{th}$ row (increasing from the bottom) and $c^{th}$ column (increasing from the left).  Hence $(r',c')$ is \emph{above} $(r,c)$ if $r'>r$ and \emph{to the right} if $c'>c$.

A transversal of $\lambda=(\lambda_1,\ldots,\lambda_n)$ is a permutation $\pi\in\S{n}$ such that each point in the graph of $\pi$ lies inside some cell of $\lambda$ (i.e., $\pi^{-1}_i \leq \lambda_i$ for $1\leq i\leq n$).  Figure \ref{figtrans} illustrates that $\pi=45321$ is a transversal of $\lambda=(5,5,3,2,2)$.  Let $\S{\lambda}$ denote all transversals of $\lambda$.

\ifgraph
\begin{figure}
	\centering
		\includegraphics[width=.3\textwidth]{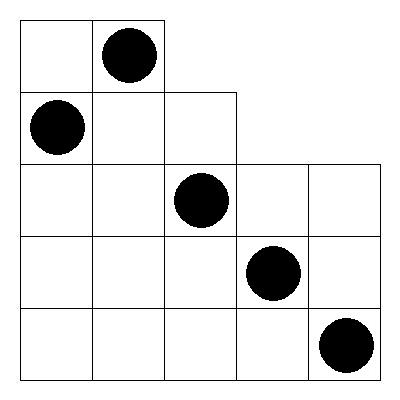}
	\caption{$\pi=45321$ is a transversal of $\lambda=(5,5,5,3,2)$.}
	\label{figtrans}
\end{figure}
\fi

Pattern containment is stricter for transversals than it is for permutations.  A transversal $\pi\in\S{\lambda}$ contains $\sigma\in \S{k}$ if there exists a subsequence $i_1<i_2<\cdots<i_k$ such that $\pi_{i_1}\pi_{i_2}\cdots\pi_{i_k} \oi \sigma$ and the cell $(\max \{\pi_{i_1},\pi_{i_2},\cdots,\pi_{i_k}\}, i_k)$ lies in $\lambda$.  In other words, the rows and columns of $\lambda$ containing $\pi_{i_1}\pi_{i_2}\cdots\pi_{i_k}$ must form a full $k\times k$ square.  In Figure \ref{figtrans} we see the transversal $45321$ in $(5,5,5,3,2)$ contains $321$ in the last three entries.  Further, the transversal $45321$ in $(5,5,5,3,2)$ avoids $231$ even though the \emph{permutation} $45321$ does not.  We let $\Sav{\lambda}{\sigma}$ denote the set of all transversals of $\lambda$ which do not contain $\sigma$, and $\sav{\lambda}{\sigma} := \#\Sav{\lambda}{\sigma}$.  Two patterns $\sigma$ and $\tau$ are called \emph{shape-Wilf-equivalent} if $\sav{\lambda}{\sigma}=\sav{\lambda}{\tau}$ for all shapes $\lambda$; we denote this $\sigma\swe\tau$.  Clearly shape-Wilf-equivalence implies Wilf-equivalence, since Wilf-equivalence considers only the shapes $\lambda$ which are $n\times n$ squares.

We adapt these concepts for even permutations as follows.  A transversal $\pi\in\S{\lambda}$ is \emph{even} if the underlying permutation $\pi$ is even.  Note that the presence/absence of an inversion is independent of $\lambda$, that is, an inversion is \emph{not necessarily} a copy of a $21$ pattern in the sense of transversals.  Let $\E{\lambda}$ be the even transversals in $\S{\lambda}$, $\Eav{\lambda}{\sigma}$ be the even transversals in $\lambda$ avoiding $\sigma$, and $\eav{\lambda}{\sigma}:=\#\Eav{\lambda}{\sigma}$.  We may do the same for odd transversals, using $\O{\lambda}$, $\Oav{\lambda}{\sigma}$, and $\oav{\lambda}{\sigma}$.  If $\eav{\lambda}{\sigma}=\eav{\lambda}{\tau}$, then we say $\sigma$ and $\tau$ are \emph{even-shape-Wilf-equivalent} and we write $\sigma\eswe\tau$.

Recall the direct sum of two permutations, $\alpha\in\S{k}$ and $\beta\in\S{\ell}$, is the length-$(k+\ell)$ permutation $\alpha_1\alpha_2\cdots\alpha_k(\beta_1+k+1)(\beta_2+k+1)\cdots(\beta_{\ell}+k+1)$.  This is most easily seen as placing $\beta$ above and to the right of $\alpha$. See Figure \ref{figsum} where we illustrate $312\oplus2413=3125746$.

\ifgraph
\begin{figure}[htb]
	\centering
		\includegraphics[height = 2 in]{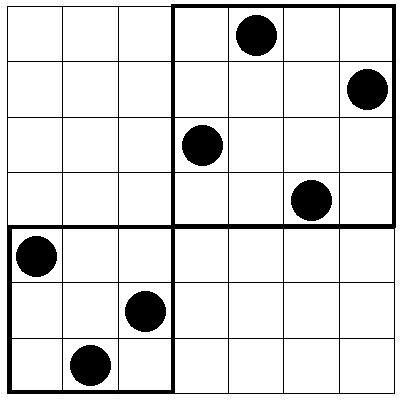}
	\caption{The direct sum $312\oplus2413=3125746$}
	\label{figsum}
\end{figure}
\fi

We now re-state Proposition 2.3 from \cite{Backelin2007} as a lemma.

\begin{lemma}[Backelin, West, and Xin, \cite{Backelin2007}] \label{lemswe}
For patterns $\alpha$ and $\beta$, $\alpha\swe\beta$ implies $\alpha\oplus\sigma \swe \beta\oplus\sigma$.
\end{lemma}

We summarize the proof here, as it will be useful for the following lemma.

\begin{proof}
  For any shape $\lambda$, let $f_{\lambda}:\Sav{\lambda}{\alpha}\to\Sav{\lambda}{\beta}$ be a 
  bijection implied by the hypothesis.  Now fix $\lambda$ and let $\pi\in\Sav{\lambda}{\alpha\oplus\sigma}$.  We will color the cells of $\lambda$ either white or gray by a two-step procedure, then transform within the white cells while leaving the gray cells fixed.  In this way we create a bijection $\Sav{\lambda}{\alpha\oplus\sigma}\to \Sav{\lambda}{\beta\oplus\sigma}$.  We illustrate these steps in Figures \ref{fig:swe-fig2} -- \ref{fig:swe-fig4}.
  \begin{enumerate}
   \item[Step 1.] Color cell $(r,c)$ white if the part of $\pi$ lying in the subboard above and to the right of it contains $\sigma$ (as a transversal).  Otherwise color $(r,c)$ gray.
   \item[Step 2.] For each point in the graph of $\pi$ which lies in a gray cell, color gray the remaining cells in its row and column.
  \end{enumerate}

\ifgraph
\begin{figure}[htb]
	\centering
		\includegraphics[height=1.5 in]{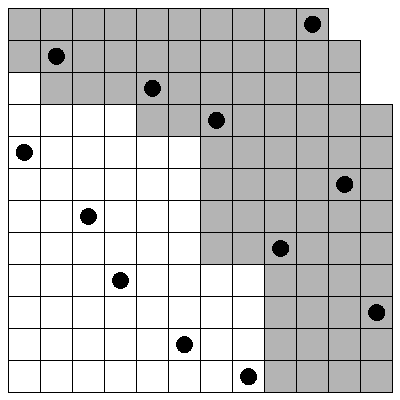}
	\caption{Executing step 1 for given $\pi$, $\lambda$ for $\sigma=12$}
	\label{fig:swe-fig2}
\end{figure}

\begin{figure}[htb]
	\centering
		\includegraphics[height=1.5 in]{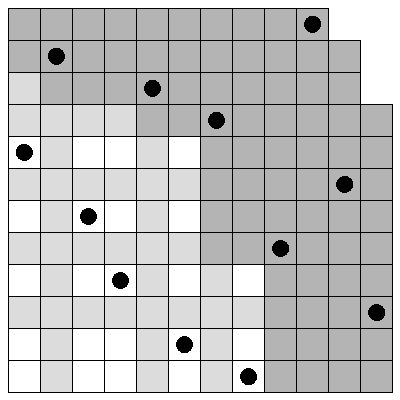}
	\caption{Executing step 2 for given $\pi$, $\lambda$ for $\sigma=12$}
	\label{fig:swe-fig3}
\end{figure}

\begin{figure}[htb]
	\centering
		\includegraphics[height=1.1 in]{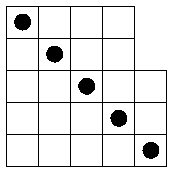}
	\caption{The resulting $\bar{\pi}$ and $\bar{\lambda}$}
	\label{fig:swe-fig4}
\end{figure}
\fi

  Denote the white cells by $\bar{\lambda}$ and the subtransversal of $\pi$ lying in $\bar{\lambda}$ by $\bar{\pi}$.  By step 2, $\bar{\lambda}$ is itself a Young diagram, and further $\bar{\pi}$ is a transversal of $\bar{\lambda}$.  Further, since $\pi$ avoids $\alpha\oplus\sigma$, step 1 implies $\bar{\pi}$ avoids $\alpha$.  We apply $f_{\bar{\lambda}}$ to $\bar{\pi}$  within the white cells, and so $f_{\bar{\lambda}}(\bar{\pi})$ avoids $\beta$.  Restoring the gray portions of $\lambda$ and $\pi$, we finish with a transversal avoiding $\beta\oplus\sigma$.
  The inverse map is identical, except that $f_{\bar{\lambda}}$ is replaced by its inverse $f_{\bar{\lambda}}^{-1}$.
\end{proof}

We are now ready to state the even-shape-Wilf-equivalence analogue.

\begin{lemma}\label{lemeswe}
For patterns $\alpha$ and $\beta$, if $\alpha\eswe\beta$ and $\alpha\swe\beta$ then $\alpha\oplus\sigma \eswe \beta\oplus\sigma$.
\end{lemma}

\begin{proof}
  We will adapt notation from Lemma \ref{lemswe} above.  Let $g_{\lambda}:\Eav{\lambda}{\alpha}\to\Eav{\lambda}{\beta}$ be a 
  bijection implied by the hypothesis.  By reasoning similar to that in Lemma \ref{lem2}, we see that we may also construct a bijection for odd transversals $h_{\lambda}:\Oav{\lambda}{\alpha}\to\Oav{\lambda}{\beta}$.

  The map $\Eav{\lambda}{\alpha\oplus\sigma} \to \Eav{\lambda}{\beta\oplus\sigma}$ is constructed in the same way as for Lemma \ref{lemswe}.  Color cells of $\lambda$ white or gray by the same rules, and isolate $\bar{\lambda}$ and $\bar{\pi}$.  Now $\bar{\pi}$ is either even or odd, so we apply the appropriate map $g_{\bar{\lambda}}$ or $h_{\bar{\lambda}}$.  Observe that these maps preserve sign and so correspond to multiplying the original $\pi$ by some even permutation.  Hence the image of the transversal $\pi$ is also even and we have our bijection.
\end{proof}

Backelin et al.\! also prove $J_t \swe I_t$, where $J_t$ is the decreasing permutation $t(t-1)\cdots 21$ and $I_t$ is the increasing permutation $12\cdots t$.  By Lemma \ref{lemswe} above, this implies the well-known ``prefix reversal'' maneuver for Wilf-equivalence, namely $12\cdots k\oplus \sigma \Swe k(k-1)\cdots 1\oplus\sigma$.  They prove\footnote{They actually provide two proofs of $J_t \swe I_t$.  Here we discuss only their first proof.} $J_t\swe I_t$ via their Proposition 3.1, that $J_t \swe F_t$ for all $t>0$, where $F_t = J_{t-1}\oplus 1=(t-1)(t-2)\cdots 21t$.  Iterating this proves $J_{t} \swe J_{t-k}\oplus I_k$ for all $0\leq k \leq t$.  They provide a bijection $\phi_t^{*}:\Sav{\lambda}{F_t}\to\Sav{\lambda}{J_t}$, which we will show preserves sign.  Here we will construct the map only; the proof of its correctness can be found in \cite{Backelin2007}.

The map from $\Sav{\lambda}{F_t}$ to $\Sav{\lambda}{J_t}$ uses the following transformation.  At its heart, it systematically converts all occurrences of $J_t$ into occurrences of $F_t$.  Suppose $\pi\in\Sav{\lambda}{J_t}$.
Then we apply the following algorithm:

\begin{algorithm}\ \ 

\begin{enumerate}
 \item[Step 1.] Find all occurrences of $J_t$ in $\pi$ (as a transversal).  If $\pi$ contains no $J_t$, then stop and return $\pi$.
 \item[Step 2.] Find the smallest letter $\pi(i_1)$ such that $\pi(i_1)$ is the leftmost letter in an copy of $J_t$.
 \item[Step 3.] Find the leftmost letter $\pi(i_2)$ such that $i_1<i_2$ and there is an occurrence of $J_t$ such that $\pi(i_1)$ and $\pi(i_2)$ are the leftmost letters.
 \item[Step 4.] Find indices $i_3< i_4< \cdots< i_t$ one by one as described in step 3.  This yields a subpermutation $\pi(i_1)\pi(i_2)\cdots\pi(i_t)$, which is a copy of $J_t$.  See Figure \ref{fig:phi1}.
 \item[Step 5.] Form a new permutation $\pi'$ by moving $\pi(i_1)$ to the ${i_t}^{th}$ position, and each other $\pi(i_j)$ to the ${i_{j-1}}^{th}$ position.  Call this transformation $\theta(\pi)=\pi'$.  Observe that $\pi'(i_1)\pi'(i_2)\cdots\pi'(i_t)$ is a copy of $F_t$.  (See Figure \ref{fig:phi2}.)
 \item[Step 6.] Return to step 1.
\end{enumerate}
\end{algorithm}

We denote a single application of steps 2 through 5 by $\phi_t(\pi)$.  As described in steps 1 and 6, we compose $\phi_t$ with itself repeatedly until all copies of $J_t$ are eliminated.  We denote this repeated composition $\phi_t^{*}$.

\ifgraph
\begin{figure}[hb]
	\centering
		\includegraphics[height=1.5 in]{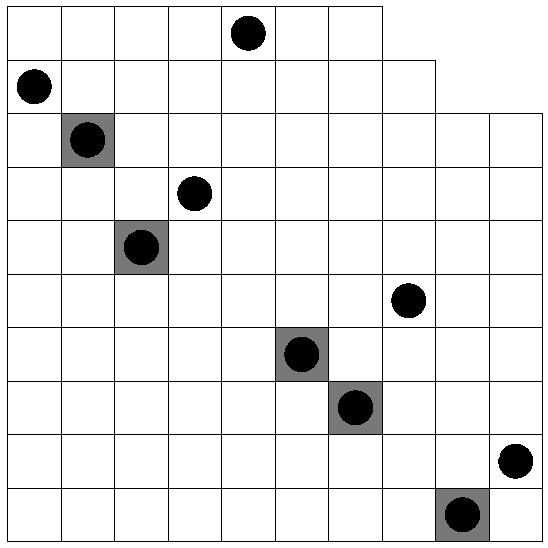}
	\caption{Selecting a copy of $J_t$}
	\label{fig:phi1}
\end{figure}

\begin{figure}[hb]
	\centering
		\includegraphics[height=1.5 in]{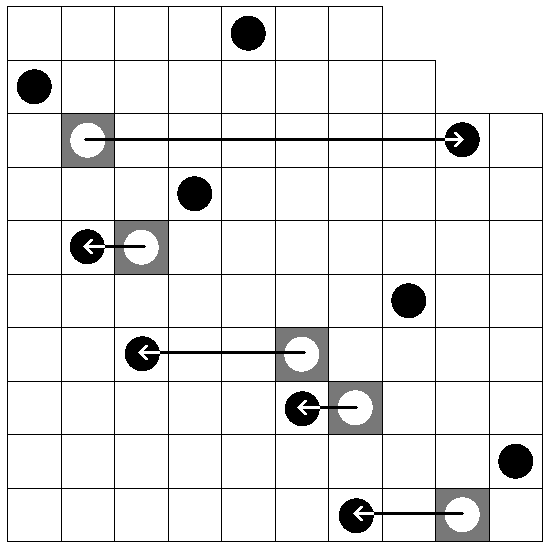}
	\caption{Applying the transformation $\theta$}
	\label{fig:phi2}
\end{figure}
\fi

For completeness we present the inverse map from $\Sav{\lambda}{J_t}$ to $\Sav{\lambda}{F_t}$.  It operates on the same principle, converting copies of $F_t$ into copies of $J_t$.

\begin{algorithm}\ \

\begin{enumerate}
 \item[Step 1.] Find all occurrences of $F_t$ in $\pi$ (as a transversal).  If $\pi$ contains no $F_t$, then stop and return $\pi$.
 \item[Step 2.] Find the largest letter $\pi(i_t)$ such that $\pi(i_t)$ is the rightmost letter in an copy of $F_t$.
 \item[Step 3.] Find the largest letter $\pi(i_{t-1})$ such that $i_{t-1}<i_{t}$ and there is an occurrence of $F_t$ such that $\pi(i_{t-1})$ and $\pi(i_{t})$ are the rightmost letters.
 \item[Step 4.] Find indices $i_{t-2}>i_{t-3} > \cdots > i_1$ one by one as described in step 3.  This yields a subpermutation $\pi(i_1)\pi(i_2)\cdots\pi(i_t)$, which is a copy of $F_t$.
 \item[Step 5.] Form a new permutation $\pi'$ by moving $\pi(i_t)$ to the ${i_1}^{th}$ position, and each other $\pi(i_j)$ to the ${i_{j+1}}^{th}$ position.  Call this transformation $\theta'(\pi)=\pi'$.  Observe that $\pi'(i_1)\pi'(i_2)\cdots\pi'(i_t)$ is a copy of $J_t$.
 \item[Step 6.] Return to step 1.
\end{enumerate}
\end{algorithm}

We denote the application of steps 2 through 5 by $\psi_t(\pi)$.  We compose $\psi_t$ with itself a certain number of times as outlined in steps 1 and 6, yielding a map $\psi_t^{*}: \Sav{\lambda}{J_t} \to \Sav{\lambda}{F_t}$.  Backelin et al.\! then show that $\phi_t$ and $\psi_t$ are inverses of one another, and hence so are $\phi_t^*$ and $\psi_t^*$.

We are now ready to prove our main result.

\begin{theorem}\label{theJtFt}
 $J_t \eswe F_t$ for all odd $t$.
\end{theorem}

\begin{proof}
 Fix $t$ odd.  The theorem follows from the claim that $\phi_t$ preserve sign.  If $\phi_t$ preserves sign, then so does $\phi_t^*$.  Hence $\phi_t^*$ restricts to the map $\phi_t^*:\Eav{\lambda}{F_t}\to\Eav{\lambda}{J_t}$.  Since $\psi_t$ is the inverse of $\phi_t$, $\psi_t^{*}$ must also preserve sign and hence we have our desired bijection.

Thus it remains to show that $\phi_t$ preserves sign when $t$ is odd.  Careful inspection reveals that the map $\theta$ in Step 5 is merely multiplying by the cycle $(i_1\,i_2\,\cdots\,i_t)$.  Since an odd cycle is an even permutation, applying $\theta$ preserves sign.
\end{proof}

It should be noted that $\phi_t$ reverses sign when $t$ is even.  This follows from the fact that $\theta$ is an even cycle and hence its application reverses sign.  Since $\phi_t$ may be composed with itself either an even or odd number of times in the application of $\phi_t^*$, however, the composition $\phi_t^*$ neither preserves nor reverses sign for the entirety of $\Eav{\lambda}{F_t}$.

The restriction that $t$ be odd prevents the iteration which implies $J_t \swe J_{t-k} \oplus I_{k}$ for all $0\leq k \leq t$.  Applying the theorem once gets us $J_t \eswe J_{t-1} \oplus 1$, at which point $t-1$ is even and the theorem no longer applies.  Note that the general prefix reversal result is not true for even-Wilf-equivalence: for example, $1234 \not\ewe 4321$.

\section{Classifications}\label{sec:class}

This section makes the classification of 4-patterns under $\ewe$ explicit, as well as the partial classifications of patterns of length 5 and 6.

\subsection{Classification of \S{4}}

With Theorem \ref{theJtFt} above and sufficient numerical computation, we may classify all patterns $\sigma \in \S{4}$.  There are eleven equivalence classes in total.  The values of each $\eav{n}{\sigma}$ are listed for $n\leq 10$ in Table \ref{tabs4}

In $\S{4}$, only two non-trivial equivalences appear.  Since $\sigma \ewe \sigma^{rc}$ we get that $3214 \ewe 1432$ and $2134 \ewe 1243$.  Applying Theorem \ref{theJtFt} and Lemma \ref{lemeswe}, we get that $3214 = J_3 \oplus 1 \ewe  F_3\oplus 1 = 2134$ to complete the class.  The reverses of these patterns comprise the other non-trivial class, as per Lemma \ref{lem2}: $3421\ewe4312\ewe4123\ewe2341$.  Thus we obtain the classification shown in Table \ref{tabs4}.  Horizontal lines separate the even-Wilf classes: patterns in the same even-Wilf class appear in adjacent rows with no separating line.

\begin{table}[htp]
	\centering
		\begin{tabular}{|l|r|r|r|r|r|r|r|r|}
\hline
$\sigma$ & $\sgn(\sigma)$ & $\eav{4}{\sigma}$ & $\eav{5}{\sigma}$ & $\eav{6}{\sigma}$ & $\eav{7}{\sigma}$ & $\eav{8}{\sigma}$ &
$\eav{9}{\sigma}$ & $\eav{10}{\sigma}$ \\
\hline
2134 & -1    & 12 & 52 & 257 & 1381 & 7885 & 47181 & 293297\\ 
3214 & -1    & 12 & 52 & 257 & 1381 & 7885 & 47181 & 293297\\ 
1243 & -1    & 12 & 52 & 257 & 1381 & 7885 & 47181 & 293297\\
1432 & -1    & 12 & 52 & 257 & 1381 & 7885 & 47181 & 293297\\ 
\hline
4312 & -1    & 12 & 52 & 256 & 1380 & 7885 & 47181 & 293293\\ 
4123 & -1    & 12 & 52 & 256 & 1380 & 7885 & 47181 & 293293\\ 
3421 & -1    & 12 & 52 & 256 & 1380 & 7885 & 47181 & 293293\\ 
2341 & -1    & 12 & 52 & 256 & 1380 & 7885 & 47181 & 293293\\ 
\hline
2314 & 1     & 11 & 51 & 257 & 1371 & 7742 & 45622 & 277826\\ 
1423 & 1     & 11 & 51 & 257 & 1371 & 7742 & 45622 & 277826\\
3124 & 1     & 11 & 51 & 257 & 1371 & 7742 & 45622 & 277826\\ 
1342 & 1     & 11 & 51 & 257 & 1371 & 7742 & 45622 & 277826\\
\hline
4132 & 1     & 11 & 51 & 255 & 1369 & 7742 & 45622 & 277836\\
3241 & 1     & 11 & 51 & 255 & 1369 & 7742 & 45622 & 277836\\
4213 & 1     & 11 & 51 & 255 & 1369 & 7742 & 45622 & 277836\\
2413 & 1     & 11 & 51 & 255 & 1369 & 7742 & 45622 & 277836\\
\hline
2413 & -1    & 12 & 52 & 256 & 1370 & 7743 & 45623 & 277831\\ 
3142 & -1    & 12 & 52 & 256 & 1370 & 7743 & 45623 & 277831\\
\hline 
1234 & 1     & 11 & 51 & 258 & 1382 & 7879 & 47175 & 293311\\
\hline
4321 & 1     & 11 & 51 & 255 & 1379 & 7879 & 47175 & 293279\\
\hline
2143 & 1     & 11 & 51 & 256 & 1380 & 7885 & 47181 & 293301\\
\hline
3412 & 1     & 11 & 51 & 257 & 1381 & 7885 & 47181 & 293289\\
\hline
1324 & -1    & 12 & 52 & 258 & 1382 & 7903 & 47393 & 296002\\ 
\hline
4231 & -1    & 12 & 52 & 255 & 1380 & 7903 & 47393 & 295948\\ 
\hline
   \end{tabular}
 \caption{The classification of $\S{4}$ into $\ewe$-classes with values of $\eav{n}{\sigma}$ for $\sigma\in\S{4}$ and $n\leq 10$.}
 \label{tabs4}
\end{table}

\subsection{Partial Classification of $\S{5}$}

The techniques of the previous section imply a partial classification of $\S{5}$.  Based on computations for $\eav{n}{\sigma}$ for $n\leq 11$, there appear to be four non-trivial equivalence classes, listed in Table \ref{tabs5}.  Each row represents one trivial equivalence class with a chosen representative.  Two rows written adjacently with no separating line are proven above to be even-Wilf-equivalent, as discussed below.  Two rows written adjacently and separated by a dotted line are conjectured to be even-Wilf-equivalent based on numerical data for $n\leq 11$.  Solid lines separate rows which are not even-Wilf-equivalent.

\begin{table}[hp]
	\centering
		\begin{tabular}{|l|r|r|r|r|r|r|}
\hline
$\sigma$ & $\sgn(\sigma)$ & $\eav{7}{\sigma}$ & $\eav{8}{\sigma}$ & $\eav{9}{\sigma}$ & $\eav{10}{\sigma}$ & $\eav{11}{\sigma}$ \\
\hline
 12345 & 1 & 2293 & 16662 & 130897 & 1095344 & 9659368  \\
 23451 & 1 & 2293 & 16662 & 130897 & 1095344 & 9659368  \\
\hdashline[1pt/5pt]
 45312 & 1 & 2293 & 16662 & 130897 & 1095344 & 9659368  \\
 34512 & 1 & 2293 & 16662 & 130897 & 1095344 & 9659368  \\
\hline
 15432 & 1 & 2289 & 16662 & 130897 & 1095344 & 9659320  \\
 54321 & 1 & 2289 & 16662 & 130897 & 1095344 & 9659320  \\
\hdashline[1pt/5pt]
 21354 & 1 & 2289 & 16662 & 130897 & 1095344 & 9659320  \\
 21543 & 1 & 2289 & 16662 & 130897 & 1095344 & 9659320  \\
\hline
 12354 & -1 & 2291 & 16662 & 130907 & 1095344 & 9659344  \\ 
 12543 & -1 & 2291 & 16662 & 130907 & 1095344 & 9659344  \\
\hdashline[1pt/5pt]
 45321 & -1 & 2291 & 16662 & 130907 & 1095344 & 9659344  \\
 34521 & -1 & 2291 & 16662 & 130907 & 1095344 & 9659344  \\
\hline
 13524 & -1 & 2290 & 16627 & 130145 & 1081965 & 9450267  \\
\hdashline[1pt/5pt]
 42531 & -1 & 2290 & 16627 & 130145 & 1081965 & 9450267  \\
\hline
		\end{tabular}
	\caption{The classification of $\S{5}$ into $\ewe$-classes with values of $\eav{n}{\sigma}$ for $\sigma\in\S{5}$ and $n\leq 11$.}
	\label{tabs5}
\end{table}

The proven equivalences are each a corollary to Theorem \ref{theJtFt} in conjunction with the symmetries in Lemmas \ref{lem2} and \ref{lem3}.  For example, $12345 \ewe 23451$ since $12345^c = 54321$, $54321 \eswe \phi_5^{*}(54321) = 43215$, and $43215^c = 23451$.  Thus we see that $\pi\mapsto \phi_5^{*}(\pi^{c})^{c}$ provides the bijection $\Eav{n}{12345} \to \Eav{n}{23451}$.
\begin{itemize}
 \item $12345 \ewe 23451$  (under $\phi_5^{*}(\pi^{c})^{c}$)
 \item $45312 \ewe 34512$  (under $\psi_3^{*}(\pi^{c})^{c}$)
 \item $15432 \ewe 54321$  (under $\phi_5^{*}(\pi)^{rc}$)
 \item $21354 \ewe 21543$  (under $\psi_3^{*}(\pi^{rc})^{rc}$)
 \item $12354 \ewe 12543$  (under $\psi_3^{*}(\pi^{rc})^{rc}$)
 \item $45321 \ewe 34521$  (under $\psi_3^{*}(\pi^{c})^{c}$)
\end{itemize}

This leaves the following conjectured equivalences:
\begin{conjecture}\label{conjs5}
  The following equivalences hold:
   \begin{itemize}
     \item $12345 \ewe 45312$
     \item $54321 \ewe 21354$
     \item $12354 \ewe 45321$
     \item $13524 \ewe 42531$
    \end{itemize}
\end{conjecture}

Observe that Lemma \ref{lem2} implies that the first and second conjectured equivalences follow from one another.

The second conjectured equivalence class contains all patterns of the form $J_r\oplus J_s$ for all $r+s=5$ and $r,s\geq 0$, together with $21354$.  The first conjectured class contains the reverses of these.  A similar pattern seems to emerge in patterns of length 7, although again conjecturally.  This suggests the following more general statement:
\begin{conjecture}
 For odd $t$, $J_r \oplus J_s \ewe J_{t}$ for any $r+s=t$.
\end{conjecture}

Also notice that the third and fourth conjectured equivalences in Conjecture \ref{conjs5} have been written in the form $\sigma\ewe\sigma^r$.  In the classical case this is trivial under the reversal map since $\Sav{n}{\sigma}^r = \Sav{n}{\sigma^r}$, but if $n=3,4 \pmod{4}$, then $\Eav{n}{\sigma}^r \cap \Eav{n}{\sigma^r}=\emptyset$ since $\Eav{n}{\sigma}^r$ contains only odd permutations by Lemma \ref{lem:sgnsymm}.

\subsection{Partial Classification of $\S{6}$}

For $\S{6}$ there are 10 non-trivial even-Wilf classes, plus two more conjectured based on numerical results.  These are listed below in Table \ref{tabs6}.  Each of these equivalences follows from Theorem \ref{theJtFt} and its symmetries.  In the classical case, classifying the length 6 patterns required an additional result provided by Stankova and West in \cite{Stankova2002}.  They prove that $312 \swe 231$, which in combination with Lemma \ref{lemswe} provides the equivalence  $312564 \Swe 231564$.  We have checked computationally for all Ferrers shapes $\lambda$ which lie in an $9\times 9$ box that $\eav{\lambda}{312} = \eav{\lambda}{231}$, and that $\eav{n}{231\oplus\alpha} = \eav{n}{312\oplus\alpha}$ for all $\alpha\in\S{1}\cup\S{2}\cup\S{3}\cup\S{4}$ and $n\leq 11$.  This naturally leads to Conjecture \ref{conjSW}, which would imply $312564 \eswe 231564$ (and $465312 \eswe 465132$ by Lemma \ref{lemeswe}).

\begin{conjecture}\label{conjSW}
  $312$ is even-shape-Wilf-equivalent to $231$.
\end{conjecture}

This analogue of the Stankova-West result, combined with those discussed in the previous sections, would complete the classification of the length 6 patterns.

\begin{table}[htp]
  \centering
  \begin{tabular}{|l|r|r|r|r|r|r|}
\hline
$\sigma$ & $\sgn(\sigma)$ & $\eav{7}{\sigma}$ & $\eav{8}{\sigma}$ & $\eav{9}{\sigma}$ & $\eav{10}{\sigma}$ & $\eav{11}{\sigma}$ \\
\hline
543216 & 1 & 2501 & 19713 & 172417 & 1645790 & 16917552 \\
432156 & 1 & 2501 & 19713 & 172417 & 1645790 & 16917552 \\
\hline
612345 & -1 & 2502 & 19713 & 172417 & 1645800 & 16917562 \\
651234 & -1 & 2502 & 19713 & 172417 & 1645800 & 16917562 \\
\hline
213564 & -1 & 2502 & 19714 & 172392 & 1644933 & 16895077 \\
321564 & -1 & 2502 & 19714 & 172392 & 1644933 & 16895077 \\
\hline
465312 & 1 & 2501 & 19714 & 172392 & 1644930 & 16895074 \\
465123 & 1 & 2501 & 19714 & 172392 & 1644930 & 16895074 \\
\hline
213456 & -1 & 2502 & 19714 & 172418 & 1645799 & 16917561 \\
321456 & -1 & 2502 & 19714 & 172418 & 1645799 & 16917561 \\
\hline
654312 & 1 & 2501 & 19714 & 172418 & 1645791 & 16917553 \\
654123 & 1 & 2501 & 19714 & 172418 & 1645791 & 16917553 \\
\hline
213546 & 1 & 2501 & 19712 & 172417 & 1645814 & 16918707 \\
321546 & 1 & 2501 & 19712 & 172417 & 1645814 & 16918707 \\
\hline
645312 & -1 & 2502 & 19712 & 172417 & 1645838 & 16918725 \\
645123 & -1 & 2502 & 19712 & 172417 & 1645838 & 16918725 \\
\hline
213465 & 1 & 2501 & 19713 & 172417 & 1645791 & 16917553 \\
321465 & 1 & 2501 & 19713 & 172417 & 1645791 & 16917553 \\
321654 & 1 & 2501 & 19713 & 172417 & 1645791 & 16917553 \\
\hline
564312 & -1 & 2502 & 19713 & 172417 & 1645799 & 16917561 \\
564123 & -1 & 2502 & 19713 & 172417 & 1645799 & 16917561 \\
456123 & -1 & 2502 & 19713 & 172417 & 1645799 & 16917561 \\
\hline
231564 & 1 & 2501 & 19716 & 172388 & 1644575 & 16882865 \\
\hdashline[1pt/5pt]
312564 & 1 & 2501 & 19716 & 172388 & 1644575 & 16882865 \\
\hline
465132 & -1 & 2502 & 19716 & 172388 & 1644588 & 16882878 \\
\hdashline[1pt/5pt]
465213 & -1 & 2502 & 19716 & 172388 & 1644588 & 16882878 \\
\hline
\end{tabular}
\caption{The classification of $\S{6}$ into $\ewe$-classes with values of $\eav{n}{\sigma}$ for $\sigma\in\S{6}$ and $n\leq 11$.}
\label{tabs6}
\end{table}

\clearpage
\section{Conclusions and Future Directions}

In this paper we have established the foundation for a theory of even-Wilf-equivalence, parallel to the classical theory of Wilf-equivalence.  As with involution-Wilf-equivalence, the general trend appears to be that results in Wilf-equivalence have weaker versions for even-Wilf-equivalence.  For example, $\sigma \not\ewe \sigma^{r}$ but if $\sigma \ewe \tau$ then $\sigma^{r} \ewe \tau^{r}$.
Similarly, we have prove an $\ewe$-analogue to Backelin et al.'s result that $J \Swe F_t$; this requires $t$ to be be odd.
These results allow us to classify $\S{4}$ according to even-Wilf-equivalence, and to partially classify $\S{k}$ for larger $k$.

Known even-Wilf-equivalences, and the ones we conjecture above based on numerical results, suggest that even-Wilf-equivalence is a refinement of Wilf-equivalence; we thus make the following conjecture:
\begin{conjecture}
 If $\sigma \ewe \tau$, then $\sigma \Swe \tau$.
\end{conjecture}
We note that the analogue (which has not been formally conjectured, but which motivated aspects of~\cite{jm11ac}) for involution-Wilf-equivalence\footnote{Two patterns $\sigma$ and $\tau$ are said to be involution-Wilf-equivalent if $\#\bigl(\Sav{n}{\sigma} \cap \I{n}\bigr) = \#\bigl(\Sav{n}{\tau} \cap \I{n}\bigr)$ for all $n$, where $\I{n}$ is the set of involutions of length $n$.} remains open.

Examining the equivalence classes under $\ewe$ suggests that even-Wilf-equivalence is a very strong condition.  Table \ref{tabWilf} summarizes the number of equivalence classes under classical and even-Wilf-equivalence.  Values for the number of equivalence classes under Wilf equivalence are taken from OEIS sequence A099952, \cite{A099952}.  Lower bounds for the even-Wilf-equivalence classes for 5- and 6-patterns are based on avoidance by permutations of length $n\leq 11$.  Upper bounds were obtained by assuming all conjectures above are false.

\begin{table}[htbp]\label{tabWilf}
	\centering
		\begin{tabular}{|r|c|c|c|c|c|c|}
\hline		
			             $n$ & 1 & 2 & 3 & 4 &    5      &    6\\
\hline
			Wilf-equivalence & 1 & 1 & 1 & 3 &   16      &   91 \\
\hline
 even-Wilf-equivalence & 1 & 1 & 2 &11 & $[35, 39]$ & $ \{216, 218\}$  \\
\hline
		\end{tabular}
		\caption{The number of equivalence classes for patterns of length $n$.}
\end{table}

There are many more trivial equivalence classes under $\ewe$ than in the classical case.  A possible weakening of even-Wilf-equivalence is perhaps to require that $\eav{n}{\sigma}=\eav{n}{\tau}$ only for ``most'' $n$.  For example, the results of Simion and Schmidt in \cite{Simion1985} imply that $\eav{n}{123}=\eav{n}{132}$ for any $n\neq 0 \pmod{4}$.  In other instances, data suggests pairs $(\sigma, \tau)$ such that $\eav{2n}{\sigma}=\eav{2n}{\tau}$ for all $n$.  For example, the enumeration schemes in \cite{Baxter2010} verify that $\eav{2n}{12345}=\eav{2n}{54321}$ for $n\leq 7$.  An investigation into these weakened forms of equivalence may yield a classification of patterns which more closely resembles Wilf-classification.


\end{document}